\documentclass[a4paper,11pt,twoside]{amsart}
\usepackage{amsmath, amsthm, amsfonts, amssymb}
\usepackage[utf8]{inputenc}
\usepackage[T1]{fontenc}
\usepackage{accents}
\usepackage[dvipsnames]{xcolor}
\usepackage{slashed} 
\usepackage{stmaryrd} 
\usepackage{verbatim} 
\usepackage[all, cmtip]{xy}
\usepackage{lmodern}
\definecolor{mygray}{gray}{.5}

\usepackage{paralist} 
\usepackage{enumerate} 
\usepackage{hyperref} 
\usepackage[toc, page]{appendix} 

\usepackage{chngcntr} 

\usepackage{graphicx} 
\usepackage[english]{babel}
\usepackage{microtype}
\usepackage{bm} 
\usepackage{xfrac} 
\usepackage{mathtools} 

\usepackage{verbatim}

\theoremstyle{definition}
\newtheorem{thmA}{Theorem}

\newtheorem{theorem}{Theorem}[section]
\newtheorem{lemma}[theorem]{Lemma}
\newtheorem{rem}[theorem]{Remark}
\newtheorem{prop}[theorem]{Proposition}
\newtheorem{cor}[theorem]{Corollary}
\theoremstyle{definition}
\newtheorem{definition}[theorem]{Definition}
\newtheorem{ex}[theorem]{Example}

\newcommand{\R}{\mathbb{R}}

\newcommand{\func}[5]{\ensuremath{#1 : \begin{dcases}\, #2  &\rightarrow  #3 \\ \, #4 &\mapsto  #5 \end{dcases}}}

\author{Timo Siebenand}
\address{Timo Siebenand
\newline Westf\"alische Wilhelms-Universit\"at M\"unster, Mathematisches Institut
\newline Einsteinstra\ss{}e 62, 48149 M\"unster, Germany}
\email{timo.siebenand@uni-muenster.de}

\title[Ideal structure of the Fourier-Stieltjes algebra]{On the ideal structure of the Fourier-Stieltjes algebra of certain groups}

\begin{document}

\begin{abstract}
	We determine the structure of the weak*-closed $G$-invariant ideals in the Fourier-Stieltjes 
	algebra
	$B(G)$ of certain groups $G$ by means of a $K$-theoretical obstruction. 
	The groups to which this applies are groups whose only irreducible unitary 
	representations that are not weakly contained in the left regular representation are class-one 
	representations. In particular, this is the case for the groups $\mathrm{SL}(2,\mathbb{R})$ and $\mathrm{SL}(2,\mathbb{C})$, which we consider as explicit examples.
\end{abstract}

\maketitle

\section{Introduction}
Let $G$ be a locally compact group and $p\in [1,\infty]$.
A unitary group representation $\pi\colon G \to \mathcal{U}(V)$ on a Hilbert space
$V$ is called $L^p$-representation if there exists
a dense subspace $V_0\subseteq V$ such that
for all $v,w \in V_0$, the matrix coefficient function $\pi_{v,w}\colon G \to \mathbb{C},\, s \mapsto
\langle \pi(s) v,w \rangle$ lies in $L^p(G)$. The representation $\pi$ is called $L^{p+}$-representation if
$\pi$ is an $L^{p'}$-re\-pre\-sentation for all $p'>p$.

A group C*-algebra $C^*_\mu(G)$ is a completion of $C_c(G)$ with respect 
to a C*-norm $\|\cdot \|_\mu$ on $C_c(G)$ such that
	\begin{align*}
		\|f\|_r \leq \|f\|_\mu \leq \|f\|_u
	\end{align*}
for all $f\in C_c(G)$, where $\|\cdot\|_r$ and $\|\cdot\|_u$ denote the
reduced and the universal C*-norm on $C_c(G)$, respectively. 

The group C*-algebra $C^*_\mu(G)$ is called exotic
if $\| \cdot \|_\mu$ is not equal to the reduced and not equal to the universal C*-norm.
It follows immediately that the identity map on $C_c(G)$ extends to surjective *-homomorphisms
$q:C^*(G)\to C_\mu^*(G)$ and $s:C_\mu^*(G) \to C^*_r(G)$, where $C^*(G)$ and $C^*_r(G)$
denote the universal and reduced group C*-algebra of $G$, respectively. 

In this article, we are mainly interested in the
 potentially exotic group C*-algebras
\begin{align*}
	C^*_{L^p}(G) \mbox{ and }C_{L^{p+}}^*(G),
\end{align*}
defined as the completion of the *-algebra $C_c(G)$ with respect to the C*-norms
\begin{align*}
	\|f \|_{C_{L^{p}}} &= \sup \lbrace \|\pi(f)\| \mid \pi \mbox{ is an $L^{p}$-representation}\, \rbrace \mbox{ and }\\
	\|f \|_{C_{L^{p+}}} &= \sup \lbrace \|\pi(f)\| \mid \pi \mbox{ is an $L^{p+}$-representation}\, \rbrace,
\end{align*}
respectively. The construction of these algebras essentially goes back to \cite{brownGuentner} (see also \cite{SW2}).

Another natural way to describe group C*-algebras is by means of certain subspaces of the (Banach) dual space of the universal group C*-algebra. This was  studied systematically by Kaliszewski, Landstad and Quigg \cite{KLQ-Exotic}.
The Fourier-Stieltjes algebra $B(G)$ of $G$, consisting of all matrix coefficients of the unitary group representations
of $G$, is a subalgebra of the algebra of bounded continuous complex-valued functions on $G$.
It can be canonically identified with the (Banach) dual space
$C^*(G)'$ of the universal group C*-algebra $C^*(G)$ via the pairing induced by
\begin{align*}
	\langle\varphi,f\rangle = \int_G  \varphi f \mathrm{d}\mu_G
\end{align*}
for $\varphi \in B(G)$ and $f\in C_c(G) \subseteq C^*(G)$, where $\mu_G$ denotes a Haar measure
on $G$.

Besides the algebra structure, $B(G)$ admits a canonical left and right $G$-action (see e.g.  \cite[Section 3]{KLQ-Exotic}).
A subspace $E\subseteq B(G)$ is called $G$-module if $E$ is closed under this
left and right action. A $G$-module $E\subseteq B(G)$
is called $G$-ideal if $E$ is an ideal in $B(G)$. It is proved in \cite{KLQ-Exotic} that the
weak*-closure of a $G$-module (resp. $G$-ideal) is again a $G$-module (resp. $G$-ideal).
An example of a weak*-closed $G$-ideal is $B_r(G)= C_r^*(G)'$. A weak*-closed
$G$-module $E\subseteq B(G)$ is said to be admissible if $B_r(G) \subseteq E$.
Every non-trivial weak*-closed $G$-ideal is admissible (see \cite{wiersmaRuan}).
More generally, the Banach dual space $C^*_\mu(G)'\subseteq B(G)$ of a group C*-algebra $C_\mu^*(G)$ of $G$ is
an example of an admissible weak*-closed $G$-module, and it is proven in \cite{KLQ-Exotic} that this correspondence from group C*-algebras of $G$ to admissible weak*-closed
$G$-submodules of $B(G)$ is one-to-one.

Group C*-algebras corresponding to $G$-ideals are of particular importance. 
This can be explained by the fact that these group C*-algebras are stable under certain invariants.
An illustration of this fact, and one of the major tools in this article, is given by the following result.
\begin{theorem}[\cite{bew1}]\label{thm:iso_Ktheory}
	Let $G$ be a K-amenable second countable locally compact group and $C^*_\mu(G)$ a group C*-algebra
	such that $C^*_\mu(G)'$ is an ideal in $B(G)$. Then the canonical
	quotient maps $q:C^*(G)\to C^*_\mu(G)$ and $s:C^*_\mu(G) \to C^*_r(G)$
	induce KK-equivalences.
\end{theorem}
Recall that a second countable locally compact group $G$ is K-amenable if the identity
in the $G$-equivariant KK-ring $KK^G(\mathbb{C},\mathbb{C})$ is represented by a $G$-equivariant Kasparov module such that the $G$-representation is weakly contained in the left regular representation of $G$ \cite{cuntz_1982}, \cite{julgValette}.

Since Theorem \ref{thm:iso_Ktheory} and its proof are the main motivation for this article, we will outline the strategy of the proof in the following:

The theorem is a special case of a more general theorem (see \cite[Theorem 6.6]{bew1}) on correspondence crossed product functors and generalizes a result of Julg and Valette  (see \cite[Proposition 3.4]{julgValette}).
The major point in the proof of Julg and Valette's result is the existence of the descent homomorphism (a functor from the category $KK^G$ to the category $KK$) constructed by Kasparov in \cite[Theorem 3.11]{kasparov_1995} for the universal crossed product functor. Buss, Echterhoff and Willett introduced in \cite{bew1} the so-called correspondence crossed product functors and generalized  Kasparov's construction (see \cite[Proposition 6.1]{bew1}) for these correspondence crossed product functors, which enabled them to prove a result analogous to \cite[Proposition 3.4]{julgValette} (see \cite[Theorem 6.6]{bew1}). 

In order to finally conclude the theorem as it is formulated above, it must be noted that for every group C*-algebra $C_\mu^*(G)$ corresponding to a $G$-ideal in the Fourier-Stieltjes algebra, there is a correspondence crossed product functor $-\rtimes_\mu G$ with $\mathbb{C}\rtimes_\mu G = C_\mu^*(G)$ \cite{bew1}. On the other hand, it is actually a necessary condition for a correspondence crossed product functor $-\rtimes_\mu G$ that $(\mathbb{C}\rtimes_\mu G)'$ is an $G$-ideal in the Fourier-Stieltjes algebra \cite[Corollary 5.7]{bew1}. 

The proof of Theorem \ref{thm:iso_Ktheory}, besides its elegance, is of a general nature and leaves open to a certain extent the necessity of the ideal assumption of $C^*_\mu(G)'$ in the Fourier-Stieltjes algebra $B(G)$. Although it is not difficult to find counterexamples that justify a certain necessity of this assumption (see e.g. \cite[Example 4.12]{bew0}), I hope to further emphasize this aspect. The main result of this article is the following theorem.
\begin{thmA}\label{thmA}
	Let $G = \mathrm{SL}(2,\mathbb{F})$ for $\mathbb{F}= \mathbb{R}, \, \mathbb{C}$,
	and let $C^*_\mu(G)$ be a group C*-algebra of $G$ 
	such that $C^*_\mu(G)' \unlhd B(G)$ is a $G$-ideal in $B(G)$. 
	Then there exists a unique element $p\in [2,\infty]$ such that $B_{L^{p+}}(G) = C^*_\mu(G)'$, 	
	where $B_{L^{p+}}(G):= C^*_{L^{p+}}(G)'$.
\end{thmA}
\begin{rem}
 Theorem \ref{thmA} was proved for $\mathrm{SL}(2,\mathbb{R})$ in \cite[Theorem 7.3]{wiersma} by means of an explicit representation theoretic argument. This argument can be modified to work for $\mathrm{SL}(2,\mathbb{C})$ (see \cite{dabeler}]). The proof presented in this article uses less about the representation theory of the groups. Also, the proof provides a more general strategy, which can be applied to groups whose non-tempered irreducible unitary representations are class-one (see Section \ref{sec:special_groups_hecke_alg} for the definition of a non-tempered representation and Section \ref{sec:prelim} for the definition of a class-one representation).
Exotic group C*-algebras of more general Lie groups have been studied in \cite{SW2} and \cite{deLaat_2019}. The results presented here complement these, in the sense that they provide, for certain groups, a full description of the weak*-closed $G$-ideals in the Fourier-Stieltjes algebra of the group.
\end{rem}
\section*{Acknowledgements}
I thank Tim de Laat and Siegfried Echterhoff for interesting discussions and useful comments.

The author is supported by the Deutsche Forschungsgemeinschaft under Germany's Excellence Strategy - EXC 2044 - 390685587, Mathematics M\"unster: Dynamics - Geometry - Structure.

\section{Preliminaries}\label{sec:prelim}
Let $G$ be a locally compact group and $K\leqslant G$ a compact subgroup. In this section, we assume $(G,K)$ to be a Gelfand pair, i.e.
\[
	C_c(K\backslash G/ K) = \lbrace f \in C_c(G) \mid f(ksk') =f(s)\, \forall s\in G,\, \forall k,\,k'\in K \rbrace
\]
is a commutative involutive subalgebra of the involutive group algebra $C_c(G)$ of $G$.

Recall that a function $\varphi\in C(K\backslash G/K)$ with $\varphi \neq 0$ is called spherical if it satisfies 
\begin{align*}
	\varphi(s)\varphi(t) = \int_K \varphi(skt)\,\mathrm{d}\mu_K(k)
\end{align*}
for all $s,t\in G$, where $\mu_K$ denotes the normalized Haar measure on $K$.
We denote the set of  positive definite spherical functions by $\mathcal{S}(K\backslash G/K)$. 

An irreducible unitary representation $\pi: G \to \mathcal{U}(V)$ is called a class-one representation
if the linear subspace $V^K$ of $K$-invariant vectors is non-trivial. In that case the dimension
of $V^K$ is equal to 1 and for a normalized $K$-invariant vector $v\in V^K$, the function
$\langle \pi(\cdot) v, v \rangle$ is a positive definite spherical function.  On the other hand,
if $\varphi \in \mathcal{S}(K\backslash G/K)$ is a positive definite spherical function on $G$, then
the GNS-construction $(L^2(G,\varphi),\pi_\varphi)$ of $\varphi$ defines a
class-one representation of $G$. For more details we refer the reader to \cite{van2009introduction}.
\section{Groups whose non-tempered irreducible unitary representations are class-one}\label{sec:special_groups_hecke_alg}

Let $(G,K)$ be a Gelfand pair consisting
of a locally compact group $G$ and a compact subgroup $K$ of $G$. The spherical unitary representation theory generally contains much information about the group $G$.
In the following, however, we will focus on the, from the point of view of representation theory,  rather pathological assumption that the space $\widehat{G}\setminus \widehat{G}_r$ consists of class-one representations, where $\widehat{G}$ is the unitary dual of $G$ and $\widehat{G}_r$ is the subset of $\widehat{G}$ consisting of all irreducible representations weakly contained in the left regular representation of $G$. For this reason we start the section with some examples.

Motivated by the the representation theory of semisimple Lie groups, a unitary representation of $G$ is said to be tempered if it is weakly contained in the left-regular representation $\lambda_G$. Thus, the non-tempered irreducible unitary representations of $G$ are the elements of $\widehat{G}\setminus \widehat{G}_r$.
\begin{ex}\label{ex:sphericalGrps}
	The first and for our purpose most important examples of groups $G$ which satisfy the property that $\widehat{G}\setminus \widehat{G}_r$ consists of class-one representations are the groups $\mathrm{SL}(2,\mathbb{R})$ and $\mathrm{SL}(2,\mathbb{C})$. The property follows in these examples directly by the well-known classification of the unitary dual of these groups, as can be found for example in \cite[Chapter XVI §1]{knapp}.
		
	Further examples are given by certain subgroups of automorphism groups $\mathrm{Aut}(T)$ of a locally finite homogeneous  tree $T$. 
	In the case of a closed subgroup $G$ of $\mathrm{Aut}(T)$ which has the independence property and acts transitively on the boundary of $T$, it can be shown that certain vector states of irreducible representations of $G$, which are not class-one, have compact support \cite[Lemma 19]{amann}, \cite[Chapter III, Proposition 3.2]{figaharmonic}  (see also \cite{olshanskii}). This implies in particular that $\widehat{G}\setminus \widehat{G}_r$ consists of $L^2$-representations.
\end{ex}
The assumption that the non-tempered irreducible unitary representations of $G$ consists of class-one representations makes it particularly easy to determine the K-theory of group C*-algebras of $G$ relative to $C^*_r(G)$. In this context, the commutative C*-algebra
$\mathcal{H}(K\backslash G/K)$, defined as
the closure of $C_c(K\backslash G/K)$ in the universal group C*-algebra $C^*(G)$ of $G$, is of particular interest. More generally, we make the following definition.
\begin{definition}
	For a general group C*-algebra $C_\mu^*(G)$ we write 
	$\mathcal{H}_\mu(K\backslash G/K)$ for the commutative C*-algebra 
	$\overline{C_c(K\backslash G/K)}\subseteq C^*_\mu(G)$.
\end{definition}

The following observation, and the main result of this section, specifies the relation between the K-theory of a group C*-algebra $C_\mu^*(G)$ of $G$ and of its commutative subalgebra $\mathcal{H}_\mu(K\backslash G/K)$.
\begin{lemma}\label{prop:iso_Ktheory}
	Let $(G,K)$ be a Gelfand pair consisting of a second countable locally compact group $G$ and a compact subgroup $K$ of $G$. 
	Suppose that $\widehat{G}\setminus \widehat{G}_r $ consists of class-one representations. Let $C_\mu^*(G)$ be a group C*-algebra and $\mathcal{H}_\mu(K\backslash G/K)$ as above. Then the canonical quotient map
	$s: C_\mu^*(G) \to C_r^*(G)$ induces an isomorphism in K-theory if and only if
	the canonical quotient map $s\vert:\mathcal{H}_\mu(K\backslash G/K)\to \mathcal{H}_r(K\backslash G/K)$, i.e. the restriction of $s$ to $\mathcal{H}_\mu(K\backslash G/K)$ and
	 $\mathcal{H}_r(K\backslash G/K)$, does.
\end{lemma}

\begin{proof}
Let $\iota_G: G \to \mathcal{UM}(C^*(G))$ be the universal unitary representation of $G$, let $q:C^*(G)\to C^*_\mu(G)$ be the canonical quotient map, and $\iota_{G,\mu} = \overline{q}\circ \iota_G$, where $\overline{q}:\mathcal{M}(C^*(G))\to \mathcal{M}(C^*_\mu(G))$ is the unique extension of $q$ to a *-homomorphism. Furthermore, let
$
	p_K = \iota_{G,\mu}(\mu_K)\in \mathcal{M}(C^*_\mu(G))
$
be the orthogonal projection given by 
$p_K (x)= \int_K \iota_{G,\mu}(k)x \mathrm{d}\mu_K(k)$ for $x\in C^*_\mu(G)$.

First of all note that the equality of $p_K C_c(G) p_K$ and $C_c(K\backslash G/K)$, together with the continuity of the map $C^*_\mu(G) \to C^*_\mu(G),\, b \mapsto p_K b p_K$, implies the identity
$
	\mathcal{H}_\mu(K\backslash G/K) = p_K C^*_\mu(G) p_K.
$

Therefore, the right ideal
$
	X_\mu = p_KC^*_\mu(G)
$
in $C^*_\mu(G)$
defines, in a canonical way, a partial $\mathcal{H}_\mu(K\backslash G/K)$-$C^*_\mu(G)$-imprimitivity bimodule that restricts to an $\mathcal{H}_\mu(K\backslash G/K)$-$\overline{C^*_\mu(G)p_K C^*_\mu(G)}$-imprimitivity bimodule. The spectrum of $\overline{C^*_\mu(G)p_K C^*_\mu(G)}$ identifies
with the open subset of $\widehat{C_\mu^*(G)}$ consisting of all irreducible representations of
$C^*_\mu(G)$ that do not vanish on the ideal $\overline{C^*_\mu(G)p_K C^*_\mu(G)}$. These representations are exactly
the class-one representations of $(G,K)$ that integrate to $C^*_\mu(G)$.

The assumption that $\widehat{G}\setminus\widehat{G}_r $ consists of class-one representations now implies that the bimodule $X_\mu$ restricts to 
a $\ker s\vert$-$\ker s$ imprimitivity bimodule $(X_\mu)_{\ker s}=\overline{ X_\mu \ker s}$. 
Lemma \ref{prop:iso_Ktheory} is therefore an immediate consequence of the 
six-term-exact-sequence in K-theory and the fact that $(X_\mu)_{\ker s}$ induces an isomorphism in K-theory.
\end{proof}
The following proposition is  well known. For the convenience of the reader we provide the proof.
\begin{prop}\label{prop:sph_func_char}
	Let $\Delta(\mathcal{H}(K\backslash G /K))$ be the Gelfand spectrum of 
	the commutative C*-algebra $\mathcal{H}(K\backslash G /K)$.
	
	The map from $\mathcal{S}(K\backslash G/K)$ to $\Delta(\mathcal{H}(K\backslash G/K))$ sending
a positive definite spherical function $\varphi$ to the character $\chi_\varphi$ 
given by 
\begin{align*}
	\chi_\varphi(f) =  \int f(s)\varphi(s^{-1})\, \mathrm{d} \mu_G(s)
\end{align*}
for $f\in C_c(K\backslash G/K)$ establishes a bijection, and, after equipping $\mathcal{S}(K\backslash G/K)$ with the topology of uniform convergence on compact subsets of $G$, even a homeomorphism.
\end{prop}
\begin{proof}
	We adopt some of the notations from the previous proof. So let $Y = C^*(G)p_K$ be the partial $C^*(G)$-$\mathcal{H}(K\backslash G /K)$-imprimitivity bimodule. 
	First of all, we will describe the inverse of the above mapping.
	Note that for a character $\chi\in \mathcal{H}(K\backslash G/K)$, $Y-\mathrm{Ind}\, \chi = Y\otimes_\chi \mathbb{C}$ is an irreducible representation $(V,\pi_*)$ of $C^*(G)$, where $V=Y\otimes_\chi \mathbb{C}$ denotes the balanced tensor product of the right Hilbert $\mathcal{H}(K\backslash G/K)$-module $Y$ and the *-representation $(\mathbb{C},\chi)$ of $\mathcal{H}(K\backslash G/K)$, 
	a Hilbert space with a canonical left multiplication $\pi_*:C^*(G)\to \mathcal{L}(V)$ by $C^*(G)$. Let $(V,\pi)$ denotes the corresponding unitary representation of $(V,\pi_*)$.
	The restriction $(V\vert_{\mathcal{H}(K\backslash G/K)}, \pi_*\vert_{\mathcal{H}(K\backslash G/K)}) $ of $(V,\pi_*)$ to $\mathcal{H}(K\backslash G/K)$ is unitary equivalent to the character $(\mathbb{C},\chi)$, where $V\vert_{\mathcal{H}(K\backslash G/K)}$ denotes the one-dimensional Hilbert space $\pi(\mathcal{H}(K\backslash G /K))V$. Let $v\in V\vert_{\mathcal{H}(K\backslash G/K)}$ be a vector of length one. Then $\varphi_\chi = \langle \pi(\cdot)v,v\rangle$
	is a positive definite spherical function. The map $\mathcal{H}(K\backslash G/K) \to \mathcal{S}(K\backslash G/K),\, \chi \mapsto \varphi_\chi$ defines the inverse map  of $\mathcal{S}(K\backslash G/K) \to \Delta(\mathcal{H}(K\backslash G/K)) ,\, \varphi \mapsto \chi_\varphi$.
	
	The continuity of $\mathcal{S}(K\backslash G/K) \to \Delta(\mathcal{H}(K\backslash G/K)) ,\, \varphi \mapsto \chi_\varphi$ is easy to check. So it remains to show that the inverse map $\Delta(\mathcal{H}(K\backslash G/K)) \to \mathcal{S}(K\backslash G/K),\, \chi \mapsto \varphi_\chi$ is continuous. In order to prove this, let 
	$(\chi_i)_{i\in I} \in\mathcal{H}(K\backslash G/K)^I$ be a convergent net with  $\chi \in \mathcal{H}(K\backslash G/K)$ as limit point. For all $f\in C_c(G)$ we have
	\begin{align*}
		\int_G f(s) \varphi_{\chi_i}(s^{-1})\mathrm{d}\mu_G(s) &= \int_G (p_K f p_K)(s) \varphi_{\chi_i}(s^{-1})\mathrm{d}\mu_G(s)\\ &\to
		\int_G (p_K f p_K)(s)  \varphi_{\chi}(s^{-1})\mathrm{d}\mu_G(s)\\ &=\int_G f(s)  \varphi_{\chi}(s^{-1})\mathrm{d}\mu_G(s)
	\end{align*}
	as $i \to \infty$. Hence, $ \varphi_{\chi_i} \to \varphi_{\chi}$ with respect to $\sigma(L^\infty(G),L^1(G))$.
	The topology
	of uniform convergence on compact subsets of $G$ and the topology induced by $\sigma(L^\infty(G),L^1(G))$
	 coincide on $\mathcal{S}(K\backslash G /K)$ (see \cite[Proposition 6.4.2]{van2009introduction}). This proves the continuity of $\mathcal{H}(K\backslash G/K) \to \mathcal{S}(K\backslash G/K),\, \chi \mapsto \varphi_\chi$. 
\end{proof}

\section{The examples $\mathrm{SL}(2,\mathbb{R})$ and $\mathrm{SL}(2,\mathbb{C})$}
Let $\mathbb{F}$ be the real or complex number field.
As indicated in Example \ref{ex:sphericalGrps}, $\widehat{\mathrm{SL}(2,\mathbb{F})}\setminus \widehat{\mathrm{SL}(2,\mathbb{F})}_r$ consists of class-one representations.
The special linear group $\mathrm{SL}(2,\mathbb{F})$ is a double cover of the identity component of the Lorentz group 
$\mathrm{SO}_0(1,n)$ (with $n=2$ in the case of $\mathbb{F}= \mathbb{R}$ and $n=3$ in the case of $\mathbb{F}= \mathbb{C}$). Hence it is, by \cite[Lemma 5.1]{deLaat_2019}, in order to prove Theorem \ref{thmA}, equivalent to consider the groups $\mathrm{SO}_0(1,n)$ with $n=2,3$ instead of $\mathrm{SL}(2,\mathbb{F})$.  Let us therefore assume from now on $G$ to be $\mathrm{SO}_0(1,n)$ with $n=2,3$. 

We write
\begin{align*}
	K = \left\lbrace \begin{pmatrix}
	1 & 0 \\ 
	0 & k
	\end{pmatrix} \in \mathrm{SO}_0(1,n)\mid k\in \mathrm{SO}(n)\right\rbrace \leq G
\end{align*}
for the canonical maximal compact subgroup, 
\begin{align*}
	A = \left\lbrace a_t = \begin{pmatrix}
	\cosh t & 0 & \sinh t \\ 
	0 & 1 & 0 \\ 
	\sinh t & 0 & \cosh t 
	\end{pmatrix} \in  M_{n+1}(\mathbb{R})
	\mid t\in \mathbb{R} \right\rbrace,
\end{align*}
 and
\begin{align*}
	\overline{A}^+ = \lbrace a_t \in A \mid t\geq 0 \rbrace.
\end{align*}
With regard to the polar decomposition $G = K\overline{A}^+ K$, the space  $\mathcal{S}(K\backslash G/K)$ of positive definite spherical functions can, as was proved in \cite{kostant1969}, be described by the parameter space $\mathcal{P}=(0,\rho]\,\cup \, i\mathbb{R}$, where $\rho = \frac{n-1}{2}$.
More precisely, there is a one-to-one map from $\mathcal{P}=(0,\rho]\cup i\mathbb{R}$  to $\mathcal{S}(K\backslash G/K)$
sending an element $\lambda\in \mathcal{P}$ to the function $\varphi_\lambda$ uniquely determined by
\begin{align*}
	\varphi_\lambda(a_t) =  \exp\left((\lambda-\rho) t \right)
\end{align*}
for $a_t\in \overline{A}^+$ (see e.g. \cite[Lemma 5.3]{deLaat_2019}).
\begin{prop}\label{prop:parameter_space}
	The map $\Phi: \mathcal{P}\to \mathcal{S}(K\backslash G/K),\, \lambda \mapsto \varphi_\lambda$ is
	an homeomorphism.
\end{prop}
\begin{proof}
	The continuity of $\Phi$ is directly seen by the identification from Proposition \ref{prop:sph_func_char} and the theorem of the dominated convergence.
	In order to show the continuity of the inverse map of $\Phi$, note that for all $s\in G$ 
	the evaluation map $\mathrm{ev}_s: C(G) \to \mathbb{C}$ is continuous with
	respect to the topology of uniform convergence on compact subsets.
	The identity  $\Phi^{-1} - \rho= \log \circ \mathrm{ev}_{a_1}$ then implies the continuity of
	$\Phi^{-1}$. 
\end{proof}
It is well known that $G$ has the Haagerup property. Since every second countable locally compact group with the Haagerup property is K-amenable (see \cite{tu99}), Theorem \ref{thm:iso_Ktheory} is applicable.
Hence, the canonical quotient map  $s:C^*_{L^{p+}}(G)\to C^*_{r}(G)$, and therefore, by  Lemma \ref{prop:iso_Ktheory},  $s\vert: \mathcal{H}_{L^{p+}}(K\backslash G/K) \to \mathcal{H}_r(K\backslash G/K)$  induces an isomorphism in K-theory for every $p\in [2,\infty)$.
In particular, the kernel $\ker s\vert$ of $s\vert$ has trivial K-theory.

The following proposition is a reformulation of \cite[Proposition 5.4]{deLaat_2019} and describes the commutative subalgebra $\mathcal{H}_{L^{p+}}(K\backslash G/K)$ of the group C*-algebra $C^*_{L^{p+}}(G)$ of $G$ in terms of their Gelfand spectrum.
\begin{prop}\label{prop:spec_of_hecke}
	Let $p\in [2,\infty)$. Then
	\[\Delta(\mathcal{H}_{L^{p+}}(K\backslash G/K)) = \left\lbrace \varphi_\lambda \mid \lambda \in \left(0,\frac{p-2}{p}\rho \right]\cup i\mathbb{R} \right\rbrace, \]
	and 
	\[
	\Delta(\ker s\vert) =\left\lbrace \varphi_\lambda \mid \lambda \in \left(0,\frac{p-2}{p}\rho \right] \right\rbrace.
	\] 
\end{prop}
Now we are ready to prove Theorem \ref{thmA}.
\begin{proof}[Proof of Theorem \ref{thmA}]
Without loss of generality, we can assume $C_\mu^*(G)' \subseteq B(G)$ to be 
a proper ideal in $B(G)$. 
By Proposition \ref{prop:iso_Ktheory} the canonical quotient map $s\vert:\mathcal{H}_\mu(K\backslash G/K) \to \mathcal{H}_r(K\backslash G/K)$ induces
an isomorphism in K-theory. By Proposition \ref{prop:parameter_space} and \ref{prop:spec_of_hecke} there is a largest $\lambda_0 \in (0,\rho)$ such that
$\varphi_{\lambda_0} \in \Delta(\mathcal{H}_\mu(K\backslash G/K))$ but
$\varphi_{\lambda}\not\in \Delta(\mathcal{H}_\mu(K\backslash G/K))$ for all $\lambda \in (\lambda_0 , \rho	]$.
If $\Delta(\ker s\vert) \neq \lbrace \varphi_\lambda \mid \lambda \in (0,\lambda_0]\rbrace$
then $\Delta(\ker s\vert)$ contains an open compact subset, and therefore $K_0(\ker s\vert) \neq 0$, which
is a contradiction. So $\Delta(\ker s\vert) = \lbrace \varphi_\lambda \mid \lambda \in (0,\lambda_0]\rbrace$, 
which implies that $\mathcal{H}_{\mu}(K\backslash G/K) = \mathcal{H}_{L^{p+}}(K\backslash G/K)$ for some
$p\in [2,\infty)$. It is easily seen that this implies that $C^*_{\mu}(G) = C^*_{L^{p+}}(G)$.
\end{proof}
\bibliography{lit}
\bibliographystyle{amsalpha}

\end{document}